\documentclass[12pt]{amsart}
\usepackage{amsmath, amssymb, amsthm, amsfonts, mathrsfs,latexsym}
\usepackage{enumerate}

\makeatletter
\@namedef{subjclassname@2010}{%
  \textup{2010} Mathematics Subject Classification}
\makeatother

\newtheorem{theorem}{Theorem}[section] 
\newtheorem{lemma}[theorem]{Lemma}     
\newtheorem{corollary}[theorem]{Corollary}
\newtheorem{prop}[theorem]{Proposition}
\newtheorem{fact}[theorem]{Fact}

\theoremstyle{definition}

\newtheorem{remark}[theorem]{Remark}

\newcommand{\id}{\ensuremath{\emph{id} } }

\newcommand{\Z}{\ensuremath{\mathbb{Z} } }

\newcommand{\Q}{\ensuremath{\mathbb{Q} } }
\newcommand{\Ha}{\ensuremath{\mathbb{H} } }
\newcommand{\K}{\ensuremath{\mathcal{K}}}
\newcommand{\Aut}{\ensuremath{\mathrm{Aut} } }
\newcommand{\Inn}{\ensuremath{\mathrm{Inn} } }

\begin{document}

\title[Ample generics]
{Hall universal group has ample generic automorphisms}

\author[S. Song]{Shichang Song}
\address{Department of Mathematics\\School of Sciences\\Beijing Jiaotong University\\3 Shangyuancun\\Beijing 100044\\
P.\,R.\,China}
\email{ssong@bjtu.edu.cn}

\date{}

\begin{abstract}
We show that the automorphism group of Philip Hall's universal locally finite group has ample generics,
that is, it admits comeager diagonal conjugacy classes in all dimensions.
Consequently, it has the small index property, is not the union of a countable chain of non-open subgroups, and has the automatic continuity property.
Also, we discuss some algebraic and topological properties of the automorphism group of Hall universal group.
For example, we show that every generic automorphism of Hall universal group is conjugate to all of its powers, and hence has roots of all orders.
\end{abstract}


\maketitle

\section{Introduction}
A Polish group is a separable and completely metrizable topological group. The automorphism group of a countable structure equipped with the product topology is a Polish group.
We say that a Polish group $G$ has \emph{ample generics} if for every $n$ the diagonal conjugacy action of $G$ on $G^n$ by
$$g\cdot(h_1,h_2\cdots,h_n)=(gh_1g^{-1},gh_2g^{-1},\cdots,gh_ng^{-1}),$$
where $g,h_1,h_2,\cdots,h_n\in G$, has a comeager orbit.
In \cite{HHLS}, Hodges \emph{et al.} first used the concept of ample generic as a tool to study the small index property for automorphism groups of some countable structures.
Based on the work of Hodges \emph{et al.} \cite{HHLS} and Truss \cite{Truss}, Kechris and Rosendal \cite{KR} characterized when automorphism groups of Fra\"{\i}ss\'{e} structures have ample generics. In \cite{KR}, they also showed that if a Polish group $G$ has ample generics, then $G$ has the following consequences:
\begin{enumerate}
  \item The group $G$ has the small index property, that is, every subgroup of index less than $2^{\aleph_0}$ is open.
  \item The group $G$ is not the union of a countable chain of non-open subgroups.
  \item Every algebraic homomorphism from $G$ into a separable group is automatically continuous, which implies that $G$ has a unique Polish group topology.
\end{enumerate}
Now there is a growing extensive list of Polish groups that have ample generics. Just mention a few:
\begin{enumerate}
  \item The automorphism group of $\omega$-stable $\omega$-categorical structures (Hodges \emph{et al.} \cite{HHLS});
  \item The automorphism group of the random graph (Hrushovski \cite{Hrushovski}; see also Hodges \emph{et al.} \cite{HHLS});
  \item The isometry group of the rational Urysohn space (Solecki \cite{Solecki});
  \item The group of Lipschitz homeomorphisms of the Baire space (Kechris and Rosendal \cite{KR});
  \item The group of Haar measure-preserving homeomorphisms of the Cantor set (Kechris and Rosendal \cite{KR});
  \item The group of homeomorphisms of the Cantor set (Kwiatkowska \cite{Ola}).
\end{enumerate}
In this paper, we will add another group into this list, which is the automorphism group of Philip Hall's universal locally finite group.
Note that Hall universal group is a countable homogeneous structure that is not $\omega$-categorical.

We follow Kechris and Rosendal's work \cite{KR}.
Let $\mathcal{K}$ be a class of finite $L$-structures in a fixed countable signature $L$.
We say that $\mathcal{K}$ has the \emph{hereditary property} (HP) if $\mathbf{B}\in\mathcal{K}$ and $\mathbf{A}$ embeddable in $\mathbf{B}$ then $\mathbf{A}$ is isomorphic to some structure in $\K$.
We say that $\K$ has the \emph{joint embedding property} (JEP) if $\mathbf{A},\mathbf{B}\in\K$, then there is $\mathbf{C}\in\K$ such that both $\mathbf{A}$ and $\mathbf{B}$ are embeddable in $\mathbf{C}$.
We say that $\K$ has the \emph{amalgamation property} (AP) if $f\colon\mathbf{A}\to\mathbf{B}$ and $g\colon\mathbf{A}\to\mathbf{C}$, with $\mathbf{A},\mathbf{B},\mathbf{C}\in\K$, are embeddings, then there are $\mathbf{D}\in\K$ and embeddings $r\colon\mathbf{B}\to\mathbf{D}$ and $s\colon\mathbf{C}\to\mathbf{D}$ with $r\circ f=s\circ g$.
If $\K$ satisfies the HP, JEP, AP, contains only countably many structures up to isomorphism, and contains structures of arbitrarily large finite cardinality, we call $\K$  a \emph{Fra\"{\i}ss\'{e} class}. For every Fra\"{\i}ss\'{e} class $\K$, the classic Fra\"{\i}ss\'{e}'s theorem tells us that there is the unique (up to isomorphism) countably infinite $L$-structure $\mathbf{K}$ such that (i) $\mathbf{K}$ is locally finite, (ii) $\mathbf{K}$ is ultrahomogeneous, that is, every isomorphism between finite substructures of $\mathbf{K}$ extends to an automorphism of $\mathbf{K}$, (ii) $\mathrm{Age}(\mathbf{K})=\K$, where $\mathrm{Age}(\mathbf{K})$ is the class of all finite $L$-structures that can be embedded into $\mathbf{K}$.
We call such $\mathbf{K}$ the \emph{Fra\"{\i}ss\'{e} limit} of $\K$.
A countably infinite structure $\mathbf{K}$ is a \emph{Fra\"{\i}ss\'{e} structure} if it is locally finite and ultrahomogeneous.
For more background on Fra\"{\i}ss\'{e} classes and Fra\"{\i}ss\'{e} limits, see Hodges \cite[Chapter 7]{Hodges}.

For every Fra\"{\i}ss\'{e} class $\K$, Truss \cite{Truss} introduced the class $\K_p$ of all systems $\mathcal{S}=\langle \mathbf{A}, \psi\colon\mathbf{B}\to\mathbf{C}\rangle$,
where $\mathbf{A},\mathbf{B},\mathbf{C}\in\K$, $\mathbf{B},\mathbf{C}\subseteq\mathbf{A}$, and $\psi$ is an isomorphism between $\mathbf{B}$ and $\mathbf{C}$.
In the notation, $p$ is short for partial isomorphisms.
The notion of embedding in $\mathcal{K}_p$ is defined as follows:
an \emph{embedding} of $\mathcal{S}$ into $\mathcal{T}=\langle\mathbf{D},\varphi\colon\mathbf{E}\to\mathbf{F}\rangle$ is an embedding $f\colon\mathbf{A}\to\mathbf{D}$ such that
$f(\mathbf{B})\subseteq\mathbf{E}$, $f(\mathbf{C})\subseteq\mathbf{F}$, and $f\circ\psi\subseteq\varphi\circ f$.
With this notion of embedding, we may say that $\K_p$ satisfies the JEP or AP.
Similarly, let $\K_p^n$ denote the class of all systems $\mathcal{S}=\langle \mathbf{A}, \psi_1\colon\mathbf{B}_1\to\mathbf{C}_1,\cdots,\psi_n\colon\mathbf{B}_n\to\mathbf{C}_n\rangle$,
where $\mathbf{A},\mathbf{B}_i, \mathbf{C}_i\in\K$, $\mathbf{B}_i,\mathbf{C}_i\subseteq\mathbf{A}$, and $\psi_i$ is an isomorphism between $\mathbf{B}_i$ and $\mathbf{C}_i$ for each $i$. Accordingly, the notion of embedding in $\K_p^n$ is defined.

Let $\mathbf{K}$ be the Fra\"{\i}ss\'{e} limit of a Fra\"{\i}ss\'{e} class $\mathcal{K}$. Kechris and Rosendal (\cite[Theorem 6.2]{KR}) proved that $\Aut(\mathbf{K})$ has ample generics if and only if $\mathcal{K}_p^n$ has the JEP and the weak amalgamation property (WAP) for every $n$, which generalizes Truss's result. Truss \cite{Truss} showed that if for every $n$ the class $\mathcal{K}_p^n$ has the JEP and the cofinal amalgamation property (CAP), that is, it has a cofinal subclass which satisfies the AP, then $\Aut(\mathbf{K})$ has ample generics.
We say that a Fra\"{\i}ss\'{e} class has the \emph{Hrushovski property} if every system $\mathcal{S}=\langle\mathbf{A},\psi_1\colon\mathbf{B}_1\to\mathbf{C}_1,\cdots,\psi_n\colon\mathbf{B}_n\to\mathbf{C}_n\rangle\in\K_p^n$ can be extended to some system
$\mathcal{T}=\langle\mathbf{D},\varphi_1\colon\mathbf{D}\to\mathbf{D},\cdots,\varphi_n\colon\mathbf{D}\to\mathbf{D}\rangle\in\K_p^n$.
Let $\mathcal{L}^n$ denote the class of all systems $\langle\mathbf{D},\varphi_1\colon\mathbf{D}\to\mathbf{D},\cdots,\varphi_n\colon\mathbf{D}\to\mathbf{D}\rangle\in\K_p^n$.
If the Fra\"{\i}ss\'{e} class $\K$ has the Hrushovski property, then the subclass $\mathcal{L}^n$ of $\K_p^n$ is cofinal.
Further, if $\mathcal{L}^n$ has the AP, then $\K_p^n$ has the CAP. Usually, the JEP for $\K_p^n$ is straightforward.
It follows that $\mathbf{K}$ has ample generic automorphisms. Hrushovski \cite{Hrushovski} originally showed that the class of finite graphs has the Hrushovski property, and thus the automorphism group of the random graph has ample generics. Solecki \cite{Solecki}, independently Vershik, showed that the class of finite metric spaces with rational distances has the Hrushovski property, and thus the isometry group of the rational Urysohn space has ample generics. In those two cases, it is not hard to verify that $\mathcal{L}^n$ has the AP.

Philip Hall \cite{Hall} constructed a countably infinite locally finite group $\Ha$ which is determined up to isomorphism by the following two properties: (i) every finite group can be embedded in $\Ha$; (ii) any two isomorphic finite subgroups of $\Ha$ are conjugate in $\Ha$. The group $\Ha$ is known as Philip Hall's universal locally finite group.
Throughout this paper, let $\mathcal{FG}$ be the class of finite groups. The class $\mathcal{FG}$ is a Fra\"{\i}ss\'{e} class and its Fra\"{\i}ss\'{e} limit is $\Ha$. Hall's construction is totally different from Fra\"{\i}ss\'{e}'s construction, while Hall's construction replies on \cite[Lemma 1]{Hall}, which implies the Hrushovski property for the class $\mathcal{FG}$.
Thus, in order to show that $\Ha$ has ample generic automorphisms, we need only to show that the subclass $\mathcal{L}^n$, which is the class of finite groups with $n$ automorphisms, has the AP; see Theorem \ref{class of FG with automorphisms has AP}.

After this paper was put in the ArXiv.org, A.\,Ivanov informed the author that some of the results in this paper is similar to his paper \cite{Ivanov}.
We follow Kechris and Rosendal's setup, while his setup is different. The main result in \cite{Ivanov} is that $\Aut(\Ha)$ is strongly bounded, that is, every isometric action of $\Aut(\Ha)$ on a metric space has bounded orbits.

Rosendal \cite{Rosendal} considered powers and roots of the generic measure preserving homeomorphism of the Cantor space and
the generic isometry of the rational Urysohn metric space.
He proved that every generic measure preserving homeomorphism of the Cantor space and
every generic isometry of the rational Urysohn metric space are conjugate to all their powers.
For generic automorphism of Hall universal locally finite group, we also prove that it is conjugate to all of its powers,
and hence has roots of all orders.

This paper is organized as follows: In Section \ref{Section AP},
we show that the class of finite groups with $n$ automorphisms has the amalgamation property for every $n$.
In Section \ref{Section Aut(H)}, we prove that $\Aut(\Ha)$ has ample generics,
and we discuss some algebraic and topological properties of the automorphism group of Hall universal group.
In Section \ref{Section Powers}, we prove that every generic automorphism of $\Ha$ is conjugate to all of its powers,
and thus has roots of all orders.

\section{The amalgamation property}\label{Section AP}
Let $\mathcal{FG}$ be the class of finite groups.
For every $n$, let $\mathcal{L}^n$ denote the class of all systems
$\langle G, \varphi_1\colon G\to G,\cdots,\varphi_n\colon G\to G\rangle$,
where $G$ is a finite group, and $\varphi_1,\cdots,\varphi_n$ are automorphisms of $G$.
Let $\mathcal{A}=\langle A,\alpha_1\colon A\to A,\cdots,\alpha_n\colon A\to A\rangle$ and
$\mathcal{B}=\langle B,\beta_1\colon B\to B,\cdots$, $\beta_n\colon B\to B\rangle$ be two systems in $\mathcal{L}^n$.
We define the notion of embedding in $\mathcal{L}^n$ as follows:
an \emph{embedding} of $\mathcal{A}$ into $\mathcal{B}$ is an embedding $f\colon A\to B$ between groups
such that $f\circ\alpha_i=\beta_i\circ f$ and $g\circ\alpha_i=\gamma_i\circ g$ for $i=1,\cdots, n$.
Let us abuse notation and denote this embedding also by $f$.
In this section, we prove that the class $\mathcal{L}^n$ has the amalgamation property, that is,
if $f\colon\mathcal{A}\to\mathcal{B}$ and $g\colon\mathcal{A}\to\mathcal{C}$, with $\mathcal{A},\mathcal{B},\mathcal{C}\in\mathcal{L}^n$, are embeddings,
there are $\mathcal{D}\in\mathcal{L}^n$ and embeddings $r\colon\mathcal{B}\to\mathcal{D}$ and $s\colon\mathcal{C}\to\mathcal{D}$ with $r\circ f=s\circ g$.

In the case $n=0$, there is a classic result:
\begin{theorem}[{\cite{Neumann}}]\label{class of FG has AP}
 The class of finite groups has the amalgamation property.
\end{theorem}
\begin{proof}
The proof is in Section 3 of \cite{Neumann}.
\end{proof}

In the case $n>0$, we have:
\begin{theorem}\label{class of FG with automorphisms has AP}
Let $\mathcal{L}^n$ denote the class of all systems $\langle G$, $\varphi_1\colon G\to G$, $\cdots$, $\varphi_n\colon G\to G\rangle$ for every $n$,
where $G$ is a finite group, and $\varphi_1$, $\cdots$ , $\varphi_n$ are automorphisms of $G$. Then the class $\mathcal{L}^n$ has the amalgamation property.
\end{theorem}
\begin{proof}
Let $\mathcal{A}=\langle A,\alpha_1\colon A\to A,\cdots,\alpha_n\colon A\to A\rangle$,
$\mathcal{B}=\langle B,\beta_1\colon B\to B,\cdots$, $\beta_n\colon B\to B\rangle$,
and $\mathcal{C}=\langle C,\gamma_1\colon C\to C,\cdots,\gamma_n\colon C\to C\rangle$,
where $A, B, C$ are finite groups and $\alpha_i, \beta_i, \gamma_i$ are automorphisms for $i=1,\cdots, n$.
Let $f\colon\mathcal{A}\to\mathcal{B}$ and $g\colon\mathcal{A}\to\mathcal{C}$ be embeddings in $\mathcal{L}^n$
satisfying that $f\colon A\to B$ and $g\colon A\to C$ are embeddings between groups
such that $f\circ\alpha_i=\beta_i\circ f$ and $g\circ\alpha_i=\gamma_i\circ g$ for $i=1,\cdots, n$.
Let $X=\langle\alpha_1,\cdots,\alpha_n\rangle$, $Y=\langle\beta_1,\cdots,\beta_n\rangle$, and $Z=\langle\gamma_1,\cdots,\gamma_n\rangle$.
Because $A, B, C$ are finite groups, we have that $X\leq\Aut(A)$, $Y\leq\Aut(B)$, $Z\leq\Aut(C)$ are of finite order.

\textbf{Claim}: There exist surjective homomorphisms $\varphi\colon Y\to X$ and $\psi\colon Z\to X$ such that $\varphi(\beta_i)=\alpha_i$ and $\psi(\gamma_i)=\alpha_i$ for $i=1,\cdots,n$. Moreover, for all $y\in Y$ and $z\in Z$, we have that $f\circ\varphi(y)=y\circ f$ and $g\circ\psi(z)=z\circ g$.

\emph{Proof of the claim}: To make things easier, we assume that $A$ is a subset of $B$ and $f$ is simply the inclusion map.
It follows from $f\circ\alpha_i=\beta_i\circ f$ that $\beta_i\restriction_A=\alpha_i$ for $i=1,\cdots,n$.
Thus, $\beta_i$ fixes $A$ setwise for $i=1,\cdots,n$. Since $Y=\langle\beta_1,\cdots,\beta_n\rangle$, we have that $y$ fixes $A$ setwise for all $y\in Y$.
We define $\varphi(y)=y\restriction_A$ for all $y\in Y$. Clearly, $\varphi(y)\in\Aut(A)$, and $\varphi(\beta_i)=\beta_i\restriction_A=\alpha_i$ for $i=1,\cdots,n$.
Since $X=\langle\alpha_1,\cdots,\alpha_n\rangle$, we have that $\varphi$ is a surjective homomorphism from $Y$ to $X$.
Moreover, $f\circ\varphi(y)=y\circ f$ for all $y\in Y$, because $f$ is the inclusion map and $\varphi(y)=y\restriction_A$.
The case for $\psi$ is similar. This completes the proof of the claim.

Let $W=Y\times_X Z$ be the fiber product of $Y$ and $Z$ over $X$, which is the following subgroup of $Y\times Z$
$$W=Y\times_X Z=\{(y,z)\in Y\times Z\mid\varphi(y)=\psi(z)\}.$$
Clearly, $W$ is finite and the projection homomorphisms are surjectively onto $Y$ and $Z$.
Then we define a semidirect product $A\rtimes W=\{\bigl(a,(y,z)\bigr)|a\in A, (y,z)\in W\}$.
The multiplication is defined as
$$\bigl(a_1,(y_1,z_1)\bigr)*\bigl(a_2,(y_2,z_2)\bigr)=\bigl(a_1\varphi(y_1)(a_2),(y_1y_2,z_1z_2)\bigr),$$
where $a_1,a_2\in A$,  and $ (y_1,z_1), (y_2,z_2)\in W$. Note that $\varphi(y_1)=\psi(z_1)$ since $(y_1,z_1)\in W$.
Let $L$ denote $A\rtimes W$.
Similarly, we define $M=B\rtimes W$ and $N=C\rtimes W$.
For all $b_1,b_2\in B$, $c_1,c_2\in C$, and $(y_1,z_1), (y_2,z_2)\in W$, we define
$$\bigl(b_1,(y_1,z_1)\bigr)*\bigl(b_2,(y_2,z_2)\bigr)=\bigl(b_1y_1(b_2),(y_1y_2,z_1z_2)\bigr),$$
and
$$\bigl(c_1,(y_1,z_1)\bigr)*\bigl(c_2,(y_2,z_2)\bigr)=\bigl(c_1z_1(c_2),(y_1y_2,z_1z_2)\bigr).$$
Clearly, $L, M, N$ are finite groups.
Next, we define $\tilde{f}\colon L\to M$ by sending $\bigl(a,(y, z)\bigr)$ to $\bigl(f(a),(y, z)\bigr)$ for all $a\in A$ and $(y, z)\in W$.
It is easy to verify that $\tilde{f}$ is well-defined and $\tilde{f}(\id_L)=\id_M$.
For all $\bigl(a_1,(y_1,z_1)\bigr), \bigl(a_2,(y_2,z_2)\bigr)\in L$, we have
\begin{align*}
 \tilde{f}\bigl((a_1,(y_1,z_1))*(a_2,(y_2,z_2))\bigr)&=\tilde{f}\bigl(a_1\varphi(y_1)(a_2),(y_1y_2, z_1z_2)\bigr)\\
 &=\bigl(f\bigl(a_1\varphi(y_1)(a_2)\bigr),(y_1y_2, z_1z_2)\bigr)\\
 &=\bigl(f(a_1)f\bigl(\varphi(y_1)(a_2)\bigr),(y_1y_2, z_1z_2)\bigr).
\end{align*}
By the claim, $f\circ\varphi(y)=y\circ f$ for all $y\in Y=\langle\beta_1,\cdots,\beta_n\rangle$.
Then,
\begin{align*}
\tilde{f}\bigl((a_1,(y_1,z_1))*(a_2,(y_2,z_2))\bigr)&=\bigl(f(a_1)f\bigl(\varphi(y_1)(a_2)\bigr),(y_1y_2,z_1z_2)\bigr)\\
 &=\bigl(f(a_1)y_1\bigl(f(a_2)\bigr),(y_1y_2,z_1z_2)\bigr)\\
 &=\bigl(f(a_1),(y_1,z_1)\bigr)*\bigl(f(a_2),(y_2,z_2)\bigr)\\
 &=\tilde{f}\bigl(a_1, (y_1,z_1)\bigr)*\tilde{f}\bigl(a_2, (y_2, z_2)\bigr).
\end{align*}
Thus, $\tilde{f}$ is a homomorphism from $L$ to $M$.
Similarly, we define a homomorphism $\tilde{g}\colon L\to N$ by sending $\bigl(a,(y,z)\bigr)$ to $\bigl(g(a),(y, z)\bigr)$ for all $a\in A$ and $(y, z) \in W$.
Because $f$ and $g$ are embeddings, it follows that $\tilde{f}$ and $\tilde{g}$ are also embeddings.
By Theorem \ref{class of FG has AP}, the class of finite groups has the AP. Thus there exist a finite group $K$ and embeddings
$\tilde{s}\colon M\to K$ and $\tilde{t}\colon N\to K$ such that $\tilde{s}\circ\tilde{f}=\tilde{t}\circ\tilde{g}$.
Let $D=\langle \tilde{s}\bigl((B,\id_W)\bigr), \tilde{t}\bigl((C,\id_W)\bigr)\rangle$. Clearly, $D$ is a finite subgroup of $K$.
Let $\tilde{\delta_i}$ denote $(\tilde{s}\circ\tilde{f})\bigl((1,(\beta_i, \gamma_i))\bigr)\in K$ for $i=1,\cdots, n$.
By the fact that $\tilde{s}\circ\tilde{f}=\tilde{t}\circ\tilde{g}$ and definitions of $\tilde{f}$ and $\tilde{g}$, it follows that for $i=1,\cdots,n$,
\begin{align*}
\tilde{\delta_i}&=(\tilde{s}\circ\tilde{f})\bigl((1,(\beta_i, \gamma_i))\bigr)=(\tilde{t}\circ\tilde{g})\bigl((1,(\beta_i, \gamma_i))\bigr)\\
&=\tilde{s}\bigl((1,(\beta_i, \gamma_i))\bigr)=\tilde{t}\bigl((1,(\beta_i, \gamma_i))\bigr).
\end{align*}
For all $b\in B$ and $i=1,\cdots,n$, we have that
\begin{align}\label{eqtn in the proof of CAP}
\tilde{\delta_i}\tilde{s}\bigl((b,\id_W)\bigr)\tilde{\delta_i}^{-1}&=\tilde{s}\bigl((1,(\beta_i, \gamma_i))\bigr)
\tilde{s}\bigl((b,\id_W)\bigr)\tilde{s}\bigl((1,(\beta_i^{-1}, \gamma_i^{-1}))\bigr)\nonumber\\
&=\tilde{s}\bigl((1,(\beta_i, \gamma_i))*(b,\id_W)*(1,(\beta_i^{-1}, \gamma_i^{-1}))\bigr)\nonumber\\
&=\tilde{s}\bigl((1\cdot\beta_i(b),(\beta_i, \gamma_i))*(1,(\beta_i^{-1}, \gamma_i^{-1}))\bigr)\nonumber\\
&=\tilde{s}\bigl((\beta_i(b),\id_W)\bigr)\in D.
\end{align}
Similarly, for all $c\in C$ and $i=1,\cdots,n$, we have that
$\tilde{\delta_i}\tilde{t}\bigl((c,\id_W)\bigr)\tilde{\delta_i}^{-1}=\tilde{t}\bigl((\gamma_i(c),\id_W)\bigr)\in D$.
Hence, $\tilde{\delta_i}D\tilde{\delta_i}^{-1}\subseteq D$ for $i=1,\cdots,n$, and thus $\tilde{\delta_i}D\tilde{\delta_i}^{-1}=D$ because $D$ is finite.
Then we define $\delta_i\colon D\to D$ by sending $d$ to $\tilde{\delta_i}d\tilde{\delta_i}^{-1}$ for all $d\in D$ and $i=1,\cdots,n$.
Clearly, $\delta_i$ is an automorphism of $D$ for $i=1,\cdots,n$. Let $\mathcal{D}=\langle D, \delta_1\colon D\to D, \cdots, \delta_n\colon D \to D\rangle$.
We define $s\colon B\to D$ by $s(b)=\tilde{s}\bigl((b,\id_W)\bigr)$ for all $b\in B$, and define $t\colon C\to D$ by $t(c)=\tilde{t}\bigl((c,\id_W)\bigr)$ for all $c\in C$.
Then it is easy to verify that $s$ and $t$ are embeddings between groups. For all $b\in B$ and $i=1,\cdots, n$, we have that
$(\delta_i\circ s)(b)=\delta_i\bigl(s(b)\bigr)=\delta_i\Bigl(\tilde{s}\bigl((b,\id_W)\bigr)\Bigr)=\tilde{\delta_i}\tilde{s}\bigl((b,\id_W)\bigr)\tilde{\delta_i}^{-1}$.
By Equation (\ref{eqtn in the proof of CAP}), we have that $(\delta_i\circ s)(b)=\tilde{s}\bigl((\beta_i(b),\id_W)\bigr)=s\bigl(\beta_i(b)\bigr)$,
and thus $\delta_i\circ s=s\circ\beta_i$ for $i=1,\cdots,n$. Similarly, we have $\delta_i\circ t=t\circ\gamma_i$ for $i=1,\cdots,n$.
Hence, $s\colon\mathcal{B}\to\mathcal{D}$ and $t\colon\mathcal{C}\to\mathcal{D}$ are embeddings in $\mathcal{L}^n$.
Then by the fact that $\tilde{s}\circ\tilde{f}=\tilde{t}\circ\tilde{g}$, it follows that $s\circ f=t\circ g$.
Therefore, $\mathcal{L}^n$ has the AP.
\end{proof}

\section{Some properties of $\Aut(\Ha)$}\label{Section Aut(H)}
In this section, we show that $\Ha$ has ample generic automorphisms.
Also, we discuss some algebraic and topological properties of the automorphism group of Hall universal group.

Clearly, the class $\mathcal{FG}$ contains countably many isomorphism types, and contains finite groups of arbitrary size.
Also, it is easy to verify that $\mathcal{FG}$ has the HP and JEP.
By Theorem \ref{class of FG has AP}, $\mathcal{FG}$ has the AP. Thus, $\mathcal{FG}$ is a Fra\"{\i}ss\'{e} class.
The Fra\"{\i}ss\'{e} limit of $\mathcal{FG}$ is Philip Hall's universal locally finite group $\Ha$.
For $\Ha$, we have the following properties:
\begin{fact}
\begin{enumerate}
  \item The group $\Ha$ is simple.
  \item Every finite group can be embedded in $\Ha$.
  \item Any two isomorphic finite subgroups of $\Ha$ are conjugate in $\Ha$.
  \item The theory $\mathrm{Th}(\Ha)$ is not $\omega$-categorical.
\end{enumerate}
\end{fact}
\begin{proof}
(1)--(3) are proved in \cite{Hall}.

(4) Since every finite group can be embedded in $\Ha$, we have that $\Ha$ has elements of arbitrary finite orders. Hence, $S_1(\mathrm{Th}(\Ha))$ is infinite.
By the Ryll-Nardzewski Theorem, $\mathrm{Th}(\Ha)$ is not $\omega$-categorical.
\end{proof}

Philip Hall's construction of $\Ha$ is different from Fra\"{\i}ss\'{e}'s construction.
Hall's construction relies on some technical lemmas. One of them, say \cite[Lemma 1]{Hall}, implies the following Hrushovski property for the class of finite groups.

\begin{theorem}[Hrushovski property]\label{Hall}
Let $K$ be a subgroup of a finite group $G$, and $\iota\colon k\mapsto k^{\iota}$ be an isomorphism from the group $K$ onto some subgroup $K^{\iota}$ of $G$.
Then we can embed $G$ as a subgroup of some finite group $H$,
and find some $h\in H$ such that conjugation by $h$ sends every $k\in K$ to $k^h=k^{\iota}\in K^{\iota}$.
\qed
\end{theorem}

\begin{theorem}\label{main theorem}
The automorphism group of Philip Hall's universal locally finite group has ample generics.
\end{theorem}
\begin{proof}
By Theorem \ref{Hall}, the Fra\"{\i}ss\'{e} class $\mathcal{FG}$ satisfies the Hrushovski property.
Let $\mathcal{L}^n$ denote the class of all systems $\langle G$, $\varphi_1\colon G\to G,\cdots,\varphi_n\colon G\to G\rangle$ for every $n$,
where $G$ is a finite group, and $\varphi_1,\cdots,\varphi_n$ are automorphisms of $G$.
Then $\mathcal{L}^n$ is cofinal in $\mathcal{FG}_p^n$. By Theorem \ref{class of FG with automorphisms has AP}, $\mathcal{FG}_p^n$ has the CAP, and thus has the WAP.
Using direct products, it is easy to verify that $\mathcal{FG}_p^n$ has the JEP. By \cite[Theorem 6.2]{KR}, $\Ha$ has ample generic automorphisms.
\end{proof}

Following Kechris and Rosendal \cite{KR}, we have the following consequences:

\begin{corollary}
\begin{enumerate}
  \item The group $\Aut(\Ha)$ has the small index property, that is, every subgroup of index less than $2^{\aleph_0}$ is open.
  \item The group $\Aut(\Ha)$ is not the union of a countable chain of non-open subgroups.
  \item Every algebraic homomorphism from $\Aut(\Ha)$ into a separable group is automatically continuous, which implies that $\Aut(\Ha)$ has a unique Polish group topology.
\end{enumerate}
\qed
\end{corollary}


\begin{theorem}\label{prop}
\begin{enumerate}
  \item The locally finite subgroup $\Inn(\Ha)\cong\Ha$ is dense in $\Aut(\Ha)$.
  \item The group $\Aut(\Ha)$ is not simple.
  \item The group $\Aut(\Ha)$ is a complete group.
\end{enumerate}
\end{theorem}
\begin{proof}
(1) Since any two isomorphic finite subgroups of $\Ha$ are conjugate in $\Ha$, we have that $\Inn(\Ha)$ and $\Aut(\Ha)$ have the same orbits acting on $\Ha^n$.
By \cite[Lemma 4.1.5(d)]{Hodges}, $\Inn(\Ha)$ is dense in $\Aut(\Ha)$.

(2) Clearly, $\Aut(\Ha)$ is uncountable. Since $\Ha\cong\Inn(\Ha)$ is countable, $\Inn(\Ha)$ is a nontrivial normal subgroup of $\Aut(\Ha)$, and thus $\Aut(\Ha)$ is not simple.

(3) Since $\Ha$ is a nonabelian simple group, by a result of Burnside in group theory (see \cite[Theorem 13.5.10]{Robinson}), $\Aut(\Ha)$ is a complete group, that is, $Z(\Aut(\Ha))=1$, and $\Aut(\Aut(\Ha))=\Inn(\Aut(\Ha))\cong\Aut(\Ha)$.
\end{proof}

\begin{remark}
 \begin{enumerate}
  \item Kechris and Rosendal \cite{KR} showed that for a locally finite Fra\"{\i}ss\'{e} structure $\mathbf{K}$ if $\Aut(\mathbf{K})$ has a dense locally finite subgroup, then $\mathrm{Age}(\mathbf{K})$ has the Hrushovski property. Thus, Theorem \ref{prop}(1) provides another way to show that $\mathcal{FG}$ has the Hrushovski property.
  \item Macpherson and Tent \cite{MT} gave a sufficient condition for automorphism groups of countable homogeneous structures to be simple. Note that $\Aut(\Ha)$ does not satisfy their sufficient condition. Specifically, $\Aut(\Ha)$ is not transitive on $\Ha$ since $\Ha$ has elements of different orders. Also, the class of finite groups does not have the free amalgamation property.
 \end{enumerate}
\end{remark}

\section{Powers and roots of generic automorphisms}\label{Section Powers}
Rosendal \cite{Rosendal} has proved that the generic measure preserving homeomorphism of the Cantor space and
the generic isometry of the rational Urysohn metric space are conjugate to all of their powers.
In this section, we will show that every generic automorphism of Hall universal locally finite group is conjugate to all of its powers, and hence it has roots of all orders.

\begin{lemma}\label{fg=gf}
Let $A\leq B$ be two finite groups, $\alpha$ an automorphism of $A$ and $\beta$ an automorphism of $B$ leaving $A$ invariant. If $\alpha\bigl(\beta(x)\bigr)=\beta\bigl(\alpha(x)\bigr)$ for all $x\in A$, then there is a finite group $C\geq B$ and $f,g\in C$ satisfying: (i) $fg=gf$; (ii) $x^f=\alpha(x)$ for all $x\in A$; (iii) $y^g=\beta(y)$ for all $y\in B$.
\end{lemma}
\begin{proof}
 Consider semidirect products $A'=A\rtimes\langle\alpha\rangle$ and $B'=B\rtimes\langle\beta\rangle$. We may assume that $A\leq A'$ and $B\leq B'$ and let $f\in A'$ and $g\in B'$ satisfy that $x^f=\alpha(x)$ for all $x\in A$ and $y^g=\beta(y)$ for all $y\in B$.
 Clearly, $A'=A\rtimes\langle f\rangle=\langle A,f\rangle$ and $B'=B\rtimes\langle g\rangle=\langle B,f\rangle$.
 Let $H=\langle f\rangle\times\langle g\rangle$.
 We define $\varphi\colon H\to\Aut(A)$ by $\varphi(f^mg^n)(x)=x^{f^mg^n}$ for all $m,n\in\Z$, $x\in A$.
 By the assumption of $\alpha$ and $\beta$, we have that $x^{fg}=x^{gf}$ for all $x\in A$. Thus, $\varphi$ is well-defined and is homomorphic.
 Define $G=A\rtimes_{\varphi} H$, the semidirect product of $A$ and $H$ with respect to $\varphi$.
 Canonically, there are embeddings $\psi_1\colon A\rtimes\langle f\rangle\to G$ and $\psi_2\colon A\rtimes\langle g\rangle\to G$.
 We may assume that $G=\langle A, f,g\rangle$.
 Since $\iota\colon A\rtimes\langle g\rangle\hookrightarrow B'=B\rtimes\langle g\rangle$, by Theorem \ref{class of FG has AP} (Amalgamation Property)
 there is a finite group $C$  and embeddings $\rho\colon G\to C$ and $\sigma\colon B'\to C$ such that $\rho\circ\psi_2=\sigma\circ\iota$.
 We may assume that $C\geq G$ and $C\geq B'$. Then $f, g\in C$ and $C,f,g$ are desired.
\end{proof}

\begin{prop}\label{taking roots}
Let $A\leq B$ be two finite groups, $\alpha$ an automorphism of $A$, and $\beta$ an automorphism of $B$ leaving $A$ invariant such that $\alpha^n=\beta\restriction_A$ for some $n\geq 1$. Then there is a finite group $C\geq B$ and an automorphism $\gamma$ of $C$ extending $\alpha$ such that $\gamma^n\restriction_B=\beta$.
\end{prop}
\begin{proof}
By Lemma \ref{fg=gf}, there is $D\geq B$ and $f,g\in D$ such that $fg=gf$, $x^f=\alpha(x)$ for all $x\in A$, and $y^g=\beta(y)$ for all $y\in B$.
Consider $C=\underbrace{D\times\cdots\times D}_{n}$.
Define $\varphi\colon D\to C$ by $x\mapsto\bigl(x,x^f,\cdots,x^{f^{n-1}}\bigr)$. Clearly, $\varphi$ is an embedding.
Let $\overline{A}=\varphi(A)$ and $\overline{B}=\varphi(B)$, the images of $A\leq B$ in $C$.
Define $\overline{\alpha}\colon\overline{A}\to\overline{A}$ by $\overline{\alpha}(z)=\varphi\bigl(\alpha(\varphi^{-1}(z))\bigr)$ for all $z\in\overline{A}$,
and $\overline{\beta}\colon\overline{B}\to\overline{B}$ by $\overline{\beta}(z)=\varphi\bigl(\beta(\varphi^{-1}(z))\bigr)$ for all $z\in\overline{B}$.
Clearly, $\overline{\alpha}\in\Aut(\overline{A})$ and $\overline{\beta}\in\Aut(\overline{B})$.
Define $\gamma\colon C\to C$ by $\gamma\bigl((z_1,z_2,\cdots,z_{n-1},z_n)\bigr)=(z_2,z_3,\cdots,z_n,z_1^g)$.
Clearly, $\gamma\in\Aut(C)$.
For all $x\in A$, one has that
\begin{align*}
 \gamma(\varphi(x))&=\gamma(x,x^f,\cdots,x^{f^{n-1}})=\bigl(x^f,x^{f^2},\cdots,x^g\bigr)\\
&=\bigl(x^f,x^{f^2},\cdots,x^{f^n}\bigr)=\overline{\alpha}(\varphi(x)).
\end{align*}
For all $y\in B$, one has that
\begin{align*}
\gamma^n(\varphi(y))&=\gamma^n(y,y^f,\cdots,y^{f^{n-1}})=\bigl(y^g,y^{fg},\cdots,y^{f^{n-1}g}\bigr)\\
&=\bigl(y^g,y^{gf},\cdots,y^{gf^{n-1}}\bigr)=\varphi(y^g)=\varphi(\beta(y))=\overline{\beta}(\varphi(y)).
\end{align*}
The key point is the fact that $fg=gf$. This $C$ and $\gamma$ is as desired.
\end{proof}

\begin{remark}
Rosendal \cite{Rosendal} proved the counterparts of Proposition \ref{taking roots} for equidistributed finite boolean algebras and finite rational metric spaces.
His proofs were built on the free amalgamation property (FAP) of the classes he considered, but the class of finite groups does not have FAP.
Our proof is adaption of wreath product and the mapping $\gamma$ was inspired by Haipeng Qu.
\end{remark}

Now, we mimic Rosendal's proof of \cite[Theorem 12]{Rosendal} to prove that every generic automorphism of Hall universal group $\Ha$ is conjugate to all of its powers,
and hence has roots of all orders.

\begin{theorem}
 Every generic automorphism of Hall universal group is conjugate to its nth power for every $n\geq 1$.
\end{theorem}
\begin{proof}
 A basic open set in $\Aut(\Ha)$ is of the form
 $$U(f,A)=\{g\in\Aut(\Ha)\mid g\restriction_A=f\restriction_A\},$$
 where $A$ is a finite subgroup of $\Ha$ and $f$ is a partial isomorphism of $\Ha$ leaving $A$ invariant.
 Let $C$ be a comeager conjugacy class of $\Aut(\Ha)$ and take dense open sets $V_i\subseteq\Aut(\Ha)$ such that $C=\bigcap_i V_i$.
 We let $a_0,a_1,\cdots,a_m,\cdots$ enumerate the elements of $\Ha$.
 We shall define a sequence of finite subgroups  $A_0\leq A_1\leq\cdots\leq A_i\leq\cdots\leq\Ha$ and
 automorphisms $f_i, g_i$ of $A_i$ such that
 \begin{enumerate}
  \item $a_i\in A_{i+1}$,
  \item $f_{i+1}$ extends $f_i$ and $g_{i+1}$ extends $g_i$,
  \item $g_i^n=f_i$,
  \item $U(f_{i+1},A_{i+1})\subseteq V_i$ and $U(g_{i+1},A_{i+1})\subseteq V_i$.
 \end{enumerate}
We begin with $A_0=\emptyset$ and trivial automorphisms $f_0=g_0$.
Suppose $A_i, f_i, g_i$ are as defined. By Theorem \ref{Hall}, there is a finite subgroup $B\leq\Ha$
containing $a_i$ and $A_i$ such that there is an automorphism $h$ of $B$ extending $g_i$.
Since $V_i$ is dense open, we have $V_i\cap U(h,B)\neq\emptyset$.
Thus, there is a basic open set $U(k,C)\subseteq V_i\cap U(h,B)$,
where $C$ is a finite subgroup of $\Ha$ and $k$ is a partial isomorphism of $\Ha$ leaving $C$ invariant.
As $U(k,C)\subseteq U(h,B)$, we may assume that $B\leq C$ and $k$ extends $h$.
Again, as $V_i$ is dense open, there is a basic open set $U(p,D)\subseteq V_i\cap U(k^n, C)$, where $D\leq\Ha$ is a finite subgroup containing $C$,
and $p$ is a partial isomorphism of $\Ha$ leaving $D$ invariant and extending $k^n\restriction_C$.
By Proposition \ref{taking roots}, there is a finite subgroup $E\leq\Ha$ containing $D$ and an automorphism $q$ of $E$ extending $k\restriction_C$ such that $q^n$ extending $p\restriction_D$. We set $A_{i+1}=E$, $f_{i+1}=q^n$, and $g_{i+1}=q$. Clearly,
$$f_{i+1}=q^n\supseteq p\restriction_D\supseteq k^n\restriction_C\supseteq h^n\supseteq g_i^n=f_i,$$
and
$$g_{i+1}=q\supseteq k\restriction_C\supseteq h\supseteq g_i.$$
Then, we have $U(f_{i+1},A_{i+1})\subseteq U(p,D)\subseteq V_i$ and $U(g_{i+1},A_{i+1})\subseteq U(k,C)\subseteq V_i$.
Now, set $f=\bigcup_i f_i$ and $g=\bigcup_i g_i$. By the properties (1) and (2), $f$ and $g$ are both automorphisms of $\Ha$.
By (3), $g^n=f$, and by (4), $f,g\in\bigcap_i V_i=C$. Thus, $f$ and $g$ are in the same conjugacy class, and hence are mutually conjugate.
 \end{proof}

Now, following exactly Rosendal's proofs of \cite[Proposition 5 and Theorem 6]{Rosendal}, we have the following:

\begin{theorem}
 Every generic automorphism $g$ of Hall universal group $\Ha$ is conjugate to its nth power for every integer $n\neq 0$ and thus has roots of all orders.

 Moreover, there is an action of $(\Q,+)$ by automorphisms of $\Ha$ such that $g$ is the action by $1\in\Q$.
 \qed
\end{theorem}

\begin{corollary}
 Every generic automorphism of Hall universal group is of infinite order. \qed
\end{corollary}

\subsection*{Acknowledgements}\label{ackref}
I would like to thank Everett C. Dade for giving me his proof of Theorem \ref{Hall} during the ASL 2015 Annual North American meeting in Urbana, IL.
I am also grateful to Su Gao, Haipeng Qu and Christian Rosendal for conversations and comments regarding the paper.

\end{document}